\documentclass[12pt]{article}

\usepackage{enumerate}
\usepackage{amsmath}
\usepackage{amssymb,latexsym}
\usepackage{amsthm}
\usepackage{color}
\usepackage{graphicx}
\usepackage{hyperref}
\usepackage{amscd}
\usepackage{graphicx}
\usepackage{array}
\usepackage{rotating}

\DeclareMathOperator{\Tr}{Tr}
\DeclareMathOperator{\N}{N}

\title{A characterization of linearized polynomials with maximum kernel}
\author{Bence Csajb\'ok, Giuseppe Marino, Olga Polverino, Ferdinando Zullo\thanks{{The
research  was supported by
Ministry for Education, University and Research of Italy MIUR (Project
PRIN 2012 "Geometrie di Galois e strutture di incidenza") and by the Italian National
Group for Algebraic and Geometric Structures and their Applications (GNSAGA
- INdAM). The first author was supported by the J\'anos Bolyai
Research Scholarship of the Hungarian Academy of Sciences and by OTKA
Grant No. K 124950. }}}
\date{}

\newcommand{\cL}{{\mathcal L}}
\newcommand{\Lb}{{\mathbb L}}

\newcommand{\cS}{{\mathcal S}}

\newcommand{\F}{{\mathbb F}}

\newcommand{\K}{{\mathbb K}}

\newcommand{\la}{\langle}
\newcommand{\ra}{\rangle}

\newtheorem{theorem}{Theorem}[section]

\newtheorem{lemma}[theorem]{Lemma}
\newtheorem{corollary}[theorem]{Corollary}

\newtheorem{proposition}[theorem]{Proposition}
\newtheorem{result}[theorem]{Result}

\newtheorem{remark}[theorem]{Remark}

\begin{document}
\maketitle

\begin{abstract}
We provide sufficient and necessary conditions for the coefficients of a $q$-polynomial $f$ over $\F_{q^n}$ which ensure that the number of distinct roots of $f$ in $\F_{q^n}$ equals the degree of $f$. We say that these polynomials have maximum kernel. As an application we study in detail $q$-polynomials of degree $q^{n-2}$ over $\F_{q^n}$ which have maximum kernel and for $n\leq 6$ we list all $q$-polynomials with maximum kernel. We also obtain information on the splitting field of an arbitrary $q$-polynomial. Analogous results are proved for $q^s$-polynomials as well, where $\gcd(s,n)=1$.
\end{abstract}

\bigskip
{\it AMS subject classification:} 11T06, 15A04

\bigskip
{\it Keywords:} Linearized polynomials, linear transformations, semilinear transformations

\section{Introduction}

A $q$-polynomial over $\F_{q^n}$ is a polynomial of the form $f(x)=\sum_i a_i x^{q^i}$, where $a_i\in \F_{q^n}$. We will denote the set of these polynomials by $\cL_{n,q}$.
Let $\K$ denote the algebraic closure of $\F_{q^n}$. Then for every $\F_{q^n} \leq \Lb \leq \K$, $f$ defines an
$\F_q$-linear transformation of $\Lb$, when $\Lb$ is viewed as an $\F_q$-vector space. If $\Lb$ is a finite field of size $q^m$ then the polynomials of
$\cL_{n,q}$ considered modulo $(x^{q^m}-x)$ form an $\F_q$-subalgebra of the $\F_q$-linear transformations of $\Lb$. Once this field $\Lb$ is fixed, we can define the \emph{kernel} of $f$ as the kernel of the corresponding $\F_q$-linear transformation of $\Lb$, which is the same as the set of roots of $f$ in $\Lb$; and the \emph{rank} of $f$ as the rank of the corresponding $\F_q$-linear transformation of $\Lb$. Note that the kernel and the rank of $f$ depend on this field $\Lb$ and from now on we will consider the case $\Lb=\F_{q^n}$.
In this case $\cL_{n,q}$ considered modulo $(x^{q^n}-x)$ is isomorphic to the $\F_q$-algebra of $\F_q$-linear transformations of the $n$-dimensional $\F_q$-vector space $\F_{q^n}$. The elements of this factor algebra are represented by $\tilde \cL_{n,q}:=\{\sum_{i=0}^{n-1}a_i x^{q^i} \colon a_i\in \F_{q^n}\}$. For $f\in \tilde \cL_{n,q}$ if $\deg f = q^k$ then we call $k$ the \emph{$q$-degree} of $f$. It is clear that in this case the kernel of $f$ has dimension at most $k$ and the rank of $f$ is at least $n-k$.

Let $U=\la u_1,u_2,\ldots,u_k\ra_{\F_q}$ be a $k$-dimensional $\F_q$-subspace of $\F_{q^n}$.
It is well known that, up to a scalar factor, there is a unique $q$-polynomial of $q$-degree $k$, which has kernel $U$.
We can get such a polynomial as the determinant of the matrix
\[
\begin{pmatrix}
x & x^q & \cdots & x^{q^k} \\
u_1 & u_1^q & \cdots & u_1^{q^k} \\
\vdots \\
u_k & u_k^q & \cdots & u_k^{q^k}
\end{pmatrix}.
\]

The aim of this paper is to study the other direction, i.e. when a given $f\in \tilde \cL_{n,q}$ with $q$-degree $k$ has kernel of dimension $k$.
If this happens then we say that $f$ is a $q$-polynomial with \emph{maximum kernel}.

\smallskip

If $f(x) \equiv a_0x+a_1x^\sigma+\cdots+a_k x^{\sigma^k} \pmod {x^{q^n}-x}$, with $\sigma=q^s$ for some $s$ with $\gcd(s,n)=1$, then we say that $f(x)$ is a $\sigma$-polynomial (or $q^s$-polynomial) with $\sigma$-degree (or $q^s$-degree) $k$. Regarding $\sigma$-polynomials the following is known.

\begin{result}\cite[Theorem 5]{GQ2009}
\label{Gow}
Let $\mathbb{L}$ be a cyclic extension of a field $\F$ of degree $n$, and suppose that $\sigma$ generates the Galois group of $\mathbb{L}$ over $\F$. Let $k$ be an integer satisfying $1 \leq k \leq n$, and let $a_0,a_1,\ldots,a_k$ be elements of $\mathbb{L}$, not all them are zero. Then the $\F$-linear transformation defined as
\[ f(x)=a_0x+a_1x^\sigma+\cdots+a_{k}x^{\sigma^{k}} \]
has kernel with dimension at most $k$ in $\Lb$.
\end{result}

Similarly to the $s=1$ case we will say that a $\sigma$-polynomial is of \emph{maximum kernel} if the dimension of its kernel equals its $\sigma$-degree.

\medskip

Linearized polynomials have been used to describe families of $\F_q$-linear \emph{maximum rank distance codes} (\emph{MRD-codes}), i.e. $\F_q$-subspaces of $\tilde{\mathcal{L}}_{n,q}$ of order $q^{nk}$ in which each element has kernel of dimension at most $k$.
The first examples of MRD-codes found were the \emph{generalized Gabidulin codes} \cite{Delsarte,Gabidulin},
that is $\mathcal{G}_{k,s}=\langle x,x^{q^s},\ldots,x^{q^{s(k-1)}} \rangle_{\F_{q^n}}$ with $\gcd(s,n)=1$; the fact that $\mathcal{G}_{k,s}$ is an MRD-code can be shown simply by using Result \ref{Gow}.
It is important to have explicit conditions on the coefficients of a linearized polynomial characterizing the number of its roots.
Further connections with projective polynomials can be found in \cite{McGS}.

\medskip

Our main result provides sufficient and necessary conditions on the coefficients of a $\sigma$-polynomial with maximum kernel.

\begin{theorem}
\label{main}
Consider
\[f(x)=a_0x+a_1x^{\sigma}+\cdots+a_{k-1}x^{\sigma^{k-1}}-x^{\sigma^k},\]
with $\sigma=q^s$, $\gcd(s,n)=1$ and $a_0,\ldots,a_{k-1} \in \F_{q^n}$. Then $f(x)$ is of maximum kernel if and only if the matrix
	\begin{equation}\label{matrix} A=\left(
	\begin{array}{cccccccccccc}
	0 & 0 & \cdots & 0 & a_0 \\
	1 & 0 & \cdots & 0 & a_1 \\
	0 & 1 & \cdots & 0 & a_2 \\
	\vdots & \vdots & & \vdots & \vdots \\
	0 & 0 & \cdots & 1 & a_{k-1}
	\end{array}
	\right) \end{equation}
	satisfies
	\[A A^{\sigma} \cdots A^{\sigma^{n-1}}=I_k,\]
where $A^{\sigma^i}$ is the matrix obtained from $A$ by applying to each of its entries the automorphism $x\mapsto x^{\sigma^i}$ and $I_k$ is the identity matrix of order $k$.
\end{theorem}

An immediate consequence of this result gives information on the splitting field of an arbitrary $\sigma$-polynomial, cf. Theorem \ref{appl}.

\medskip

In Section \ref{k=n-2} we study in details the $\sigma$-polynomials of $\sigma$-degree $n-2$ for each $n$. For $n\leq 6$ we also provide a list of all $\sigma$-polynomials with maximum kernel cf. Sections \ref{n=4}, \ref{n=5} and \ref{n=6}. These results might yield further classification results and examples of $\F_q$-linear MRD-codes.

\section{Preliminary Results}

In this section we recall some results of Dempwolff, Fisher and Herman from \cite{DFH}, adapting them to our needs in order to make this paper self-contained.

\smallskip
Let $V$ be a $k$-dimensional vector space over the field $\mathbb{F}$ and let $T$ be a semilinear transformation of $V$. A $T$-cyclic subspace of $V$ is an $\mathbb{F}$-subspace of $V$ spanned by $\{\mathbf{v},T(\mathbf{v}),\ldots\}$ over $\mathbb{F}$ for some $\mathbf{v}\in V$, which will be denoted by $[\mathbf{v}]$. We first recall the following lemma.

\begin{lemma}\cite[Theorem 1]{DFH}\label{cyclicsubspaces}
Let $V$ be an $n$-dimensional vector space over the field $\mathbb{F}$, $\sigma$ an automorphism of $\mathbb{F}$ and $T$ an invertible $\sigma$-semilinear transformation on $V$.
Then
\[ V=[\mathbf{u}_1]\oplus \ldots \oplus [\mathbf{u}_r] \]
for $T$-cyclic subspaces satisfying $\dim[\mathbf{u}_1]\geq \dim[\mathbf{u}_2]\geq \ldots \geq \dim [\mathbf{u}_r]\geq 1$.
\end{lemma}

\begin{theorem}\label{fixsemi}
Let $T$ be an invertible semilinear transformation of $V=V(k,q^n)$ of order $n$, with companion automorphism $\sigma\in \mathrm{Aut}(\F_{q^n})$ such that $\mathrm{Fix}(\sigma)=\F_q$.
Then $\mathrm{Fix}(T)$ is a $k$-dimensional $\F_q$-subspace of $V$ and $\la \mathrm{Fix}(T) \ra_{\F_{q^n}}=V$.
\end{theorem}
\begin{proof}
First assume that the companion automorphism of $T$ is $x \mapsto x^q$ and that there exists $\mathbf{v} \in V$ such that
\[V=\langle \mathbf{v}, T(\mathbf{v}),\ldots,T^{k-1}(\mathbf{v}) \rangle_{\F_{q^n}}.\]

\noindent Following the proof of \cite[Main Theorem]{DFH}, consider the ordered basis $\mathcal{B}_T=(\mathbf{v},T(\mathbf{v}),\ldots,T^{k-1}(\mathbf{v}))$ and let $A$ be the matrix associated with $T$ with respect to the basis $\mathcal{B}_T$, i.e.
\begin{equation}\label{matrix2} A=\left(
	\begin{array}{cccccccccccc}
	0 & 0 & \cdots & 0 & \alpha_0 \\
	1 & 0 & \cdots & 0 & \alpha_1 \\
	0 & 1 & \cdots & 0 & \alpha_2 \\
	\vdots & \vdots & & \vdots & \vdots \\
	0 & 0 & \cdots & 1 & \alpha_{k-1}
	\end{array}
	\right)\in \F_{q^n}^{k\times k}, \end{equation}
where $T^k(\mathbf{v})=\sum_{i=1}^k \alpha_{i-1} T^i(\mathbf{v})$ with $\alpha_0,\ldots,\alpha_{k-1} \in \F_{q^n}$ and, since $T$ is invertible, we have $\alpha_0\neq 0$.
Denote by $\overline{T}$ the semilinear transformation of $\F_{q^n}^k$ having $A$ as the associated matrix with respect to the canonical ordered basis $\mathcal{B}_{C}=(\mathbf{e}_1,\ldots,\mathbf{e}_k)$ of $\F_{q^n}^k$ and companion automorphism $x \mapsto x^q$.
Note that $c_{\mathcal{B}_T}(\mathrm{Fix}(T))=\mathrm{Fix}(\overline{T})$, where $c_{\mathcal{B}_T}$ is the coordinatization with respect to the basis $\mathcal{B}_T$.
Also, since $T$ has order $n$, we have
\begin{equation}\label{AA^q}
A A^{q} \cdots A^{q^{n-1}}=I_k,
\end{equation}
where $A^{q^i}$, for $i \in \{1,\ldots,n-1\}$, is the matrix obtained from $A$ by applying to each of its entries the automorphism $x\mapsto x^{q^i}$.
A vector $\mathbf{z}=(z_0,\ldots,z_{k-1})\in \F_{q^n}^k$ is fixed by $\overline{T}$ if and only if
\[ \left\{ \begin{array}{llll} \alpha_0 z_{k-1}^q=z_0\\ z_0^q+\alpha_1z_{k-1}^q=z_1\\ \vdots \\ z_{k-2}^q+\alpha_{k-1}z_{k-1}^q=z_{k-1} \end{array}\right. \]
Eliminating $z_0,\ldots,z_{k-2}$, we obatin the equation
\[ \alpha_0^{q^{k-1}}z_{k-1}^{q^k}+\alpha_1^{q^{k-2}}z_{k-1}^{q^{k-1}}+\ldots+\alpha_{k-1}z_{k-1}^q-z_{k-1}=0, \]
which has $q^k$ distinct solutions in the algebraic closure $\mathbb{K}$ of $\F_{q^n}$ by the derivative test.
Each solution determines a unique vector of $\mathrm{Fix}(\overline{T})$ in $\mathbb{K}^k$.
Also, the set $\mathrm{Fix}(\overline{T})$ is an $\F_q$-subspace of $\mathbb{K}^k$ and hence $\dim_{\F_q} \mathrm{Fix}(\overline{T})=k$.
Let $\{\mathbf{w}_1,\ldots,\mathbf{w}_k\}$ be an $\F_q$-basis of $\mathrm{Fix}(\overline{T})$ and note that since $|\mathrm{Fix}(\overline{T})|=q^k$, a vector $\displaystyle \sum_{i=1}^k a_i \mathbf{w}_i$ is fixed by $\overline{T}$ if and only if $a_i \in \F_q$. This implies that $\mathbf{w}_1,\ldots,\mathbf{w}_k$ are also $\mathbb{K}$-independent.
Thus $\langle \mathrm{Fix}(\overline{T}) \rangle_{\mathbb{K}}=\mathbb{K}^k$ and $\{\mathbf{w}_1,\ldots,\mathbf{w}_k\}$ is also a $\mathbb{K}$-basis of $\mathbb{K}^k$.
Denote by $\phi$ the $\mathbb{K}$-linear transformation such that $\phi(\mathbf{w}_i)=\mathbf{e}_i$ and by $P$ the associated matrix with respect to the canonical basis $\mathcal{B}_C$, so $P\in\mathrm{GL}(k,\mathbb{K})$.
The semilinear transformation $\phi \circ \overline{T} \circ \phi^{-1}$ has companion automorphism $x \mapsto x^q$, order $n$ and associated matrix with respect to the canonical basis $P\cdot A\cdot P^{-q}$, where $P^{-q}$ is the inverse of $P$ in which the automorphism $x\mapsto x^q$ is applied entrywise.
Note that $\phi \circ \overline{T} \circ \phi^{-1}(\mathbf{e}_i)=\phi(\overline{T}(\mathbf{w}_i))=\phi(\mathbf{w}_i)=\mathbf{e}_i$, hence
\begin{equation}\label{4}
P\cdot A\cdot P^{-q}=I_k,
\end{equation}
i.e.
\begin{equation}\label{PAP-1}
P^{q}=P\cdot A.
\end{equation}
By Equations \eqref{AA^q} and \eqref{PAP-1} and using induction we get
\[ P^{q^n}= P\cdot A \cdot A^q\cdot \ldots \cdot A^{q^{n-1}}=P,  \]
i.e. $P \in \F_{q^n}^{k\times k}$.
This implies that $\mathrm{Fix}(\overline{T})$ is an $\F_q$-subspace of $\F_{q^n}^k$ of dimension $k$ and hence $\mathrm{Fix}(T)=c_{\mathcal{B}_T}^{-1}(\mathrm{Fix}(\overline{T}))$ is a $k$-dimensional subspace of $V(k,q^n)$ with the property that $\la \mathrm{Fix}(T) \ra_{\F_{q^n}}=V$.

Consider now the general case, i.e. suppose $T$ as in the statement, that is $T$ is an invertible semilinear map of order $n$ with companion automorphism $x\mapsto x^{q^s}$ and $\gcd(s,n)=1$.
Since $\gcd(s,n)=1$ there exist $l,m \in \mathbb{N}$ such that $1=sl+mn$, and hence $\gcd(l,n)=1$.
Then the semilinear transformation $T^l$ has order $n$, companion automorphism $x \mapsto x^q$ and $\mathrm{Fix}(T)=\mathrm{Fix}(T^l)$.
By Lemma \ref{cyclicsubspaces}, we may write
\[ V=[\mathbf{u}_1]\oplus \ldots \oplus [\mathbf{u}_r], \]
where $[\mathbf{u}_i]$ is a $T^l$-cyclic subspace of $V$ of dimension $m_i\geq1$, for each $i \in \{1,\ldots,r\}$, and $\sum_{i=1}^r m_i=k$.
Then we can restrict $T^l$ to each subspace $[\mathbf{u}_i]$ and by applying the previous arguments we get that $U_i=\mathrm{Fix}(T^l|_{[\mathbf{u}_i]})$ is an $\F_q$-subspace of $[\mathbf{u}_i]$ of dimension $m_i$ with the property that
$\langle U_i \rangle_{\F_{q^n}}=[\mathbf{u}_i]$.
Thus
\[\mathrm{Fix}(T)=\mathrm{Fix}(T^l)=U_1\oplus\ldots\oplus U_r\]
is an $\F_q$-subspace of dimension $k$ of $V$ with the property that $\la \mathrm{Fix}(T) \ra_{\F_{q^n}}=V$.
\end{proof}

The existence of a matrix $P \in \mathrm{GL}(k,\mathbb{K})$, with $\mathbb{K}$ the algebraic closure of a finite field of order $q$, satisfying \eqref{4} is also a consequence of the celebrated Lang's Theorem \cite{Lang} on connected linear algebraic groups. More precisely, by Lang's Theorem, since $\mathrm{GL}(k,\mathbb{K})$ is a connected linear algebraic group, the map $M \in \mathrm{GL}(k,\mathbb{K})\mapsto M^{-1}\cdot M^{q} \in  \mathrm{GL}(k,\mathbb{K})$ is onto. In Theorem \ref{fixsemi} it is proved that, if the semilinear transformation of $V(k,q^n)$ having $A$ as associated matrix has order $n$, then $P \in \mathrm{GL}(k,\F_{q^n})$.

\begin{remark}
Let $T$ be an invertible semilinear transformation of $V=V(k,q^n)$ with companion automorphism $x\mapsto x^q$ and let $\mathbb{K}$ be the algebraic closure of $\F_{q^n}$.
Denote by $\overline{T}$ the semilinear transformation of $\mathbb{K}^k$ associated with $T$ as in the proof of Theorem \ref{fixsemi}.
If $\lambda \in \mathbb{K}$, then the set $E(\lambda):=\{\mathbf{v} \in \mathbb{K}^k \colon \overline{T}(\mathbf{v})=\lambda \mathbf{v}\}$ is an $\F_q$-subspace of $\mathbb{K}^k$.
By \cite[page 293]{DFH}, it follows that $E(\lambda)=\lambda^{\frac{1}{q-1}}\mathrm{Fix}(\overline{T})$ and by \cite[Main Theorem]{DFH} $E(\lambda)$ is a $k$-dimensional $\F_q$-subspace of $\mathbb{K}^k$.
Also, when $T$ has order $n$ and $\lambda^{\frac{1}{q-1}} \in \F_{q^n}$, by Theorem \ref{fixsemi}, $E(\lambda)$ is a $k$-dimensional $\F_q$-subspace contained in $\F_{q^n}^k$ such that $\langle E(\lambda)\rangle_{\F_{q^n}}=\F_{q^n}^k$.
\end{remark}

\section{Main Results}

Now we are able to prove our main result:

\medskip

\emph{Proof of Theorem \ref{main}.}
First suppose $\dim_{\F_q} \ker f=k$. Then there exist $u_0,u_1,\ldots,u_{k-1} \in \F_{q^n}$ which form an $\F_q$-basis of $\ker f$.\\
\noindent Put $\mathbf{u}:=(u_0,u_1,\ldots,u_{k-1})\in\F_{q^n}^k$. Since $u_0,u_1,\ldots,u_{k-1}$ are $\F_q$-linearly independent, by \cite[Lemma 3.51]{LN}, we get that ${\mathcal B}:=(\mathbf{u},\mathbf{u}^{q^s},\ldots,\mathbf{u}^{q^{s(k-1)}})$ is an ordered $\F_{q^n}$-basis of $\F_{q^n}^k$. Also, $\mathbf{u}^{q^{sk}}=a_0\mathbf{u}+a_1\mathbf{u}^{q^s}+\cdots+a_{k-1}\mathbf{u}^{q^{s(k-1)}}$. It can be seen that the matrix \eqref{matrix} represents the $\F_{q^n}$-linear part of the $\F_{q^n}$-semilinear map $\overline{\sigma}\colon {\mathbf v}\in\F_{q^n}^k\mapsto {\mathbf v}^{q^s}\in\F_{q^n}^k$ w.r.t. the basis $\mathcal B$. Since $\gcd (s,n)=1$, $\overline{\sigma}$ has order $n$ and hence the assertion follows.

\smallskip

Viceversa, let $\tau$ be defined as follows

\begin{equation}\label{law}
\tau\colon
\left(
\begin{array}{c}
x_0\\
x_1\\
\vdots\\
x_{k-1}
\end{array}
\right)
\in\F_{q^n}^k\mapsto A
\left(
\begin{array}{c}
x_0\\
x_1\\
\vdots\\
x_{k-1}
\end{array}
\right)^{q^s}\in\F_{q^n}^k,
\end{equation}
where $A$ is as in \eqref{matrix} with the property $A A^{q^s}\cdots A^{q^{s(n-1)}}=I_k$. Then $\tau$ has order $n$ and, by Theorem \ref{fixsemi}, it fixes a $k$-dimensional $\F_q$-subspace $\mathcal{S}$ of $\F_{q^n}^k$ with the property that $\langle \mathcal{S} \rangle_{\F_{q^n}}=\F_{q^n}^k$.

Let $\mathcal{B}_{\mathcal{S}}=(\mathbf{s}_0,\ldots,\mathbf{s}_{k-1})$ be an $\F_q$-basis of $\cS$ and note that, since $\langle \mathcal{S} \rangle_{\F_{q^n}}=\F_{q^n}^k$, $\mathcal{B}_{\cS}$ is also an $\F_{q^n}$-basis of $\F_{q^n}^k$, then denoting by $\mathcal{B}_C$ the canonical ordered basis of $\F_{q^n}^k$, there exists a unique isomorphism $\phi$ of $\F_{q^n}^k$ such that $\phi(\mathbf{s}_i)=\mathbf{e}_i$ for each $i \in \{1,\ldots,k\}$. Then $\overline{\sigma}=\phi \circ \tau \circ \phi^{-1}$, where $\overline{\sigma}\colon {\mathbf v}\in\F_{q^n}^k\mapsto {\mathbf v}^{q^s}\in\F_{q^n}^k$. Also,
\begin{equation}\label{pow}
\overline{\sigma}^i = \phi \circ \tau^i \circ \phi^{-1},
\end{equation}
for each $i \in \{1,\ldots,n-1\}$.
Also, by \eqref{law}
\[
\begin{array}{llcll}
\tau(\mathbf{e}_0)=\mathbf{e}_1, \\
\tau(\mathbf{e}_1)=\tau^2(\mathbf{e}_0)=\mathbf{e}_2, \\
\vdots \\
\tau(\mathbf{e}_{k-1})=\tau^k(\mathbf{e}_0)=(a_0,\ldots,a_{k-1})=a_0\mathbf{e}_0+\cdots+ a_{k-1}\mathbf{e}_{k-1}. \\
\end{array}
\]
So, we get that
\[ \tau^k(\mathbf{e}_0)=a_0 \mathbf{e}_0 + a_1 \tau(\mathbf{e}_0) + \cdots + a_{k-1} \tau^{k-1}(\mathbf{e}_0), \]
and applying $\phi$ it follows that
\[ \phi(\tau^k(\mathbf{e}_0))=a_0 \phi(\mathbf{e}_0) + a_1 \phi(\tau(\mathbf{e}_0)) + \cdots + a_{k-1} \phi(\tau^{k-1}(\mathbf{e}_0)). \]
By \eqref{pow} the previous equation becomes
\[\overline{\sigma}^k(\phi(\mathbf{e}_0))=a_0 \phi(\mathbf{e}_0) + a_1 \overline{\sigma}(\phi(\mathbf{e}_0)) + \cdots + a_{k-1} \overline{\sigma}^{k-1}(\phi(\mathbf{e}_0)). \]
Put $\mathbf{u}=\phi(\mathbf{e}_0)$, then
\[ \mathbf{u}^{q^{sk}}=a_0 \mathbf{u} + a_1 \mathbf{u}^{q^s} + \cdots + a_{k-1} \mathbf{u}^{q^{s(k-1)}}.\]
This implies that $u_0, u_1,\dots,u_{k-1}$ are elements of $\mathrm{ker}\,f$, where\\ \noindent $\mathbf{u}=(u_0,\ldots,u_{k-1})$. Also, they are $\F_q$-independent since ${\mathcal B}=(\mathbf{u},\ldots,\mathbf{u}^{q^{s(k-1)}})=(\phi(\mathbf{e}_0),\ldots,\phi(\mathbf{e}_{k-1}))$ is an ordered $\F_{q^n}$-basis of $\F_{q^n}^k$. This completes the proof.
\qed

\medskip

As a corollary we get the second part of \cite[Theorem 10]{GQ2009x}, see also \cite[Lemma 3]{Sh} for the case $s=1$ and \cite{Ore} for the case when $q$ is a prime. Indeed, by evaluating the determinants in
$A A^{q^s} \cdots A^{q^{s(n-1)}}=I_k$ we obtain the following corollary.\footnote{For $x\in \F_{q^n}$ and for a subfield $\F_{q^m}$ of $\F_{q^n}$ we will denote by $\N_{q^n/q^m}(x)$ the norm of $x$ over $\F_{q^m}$ and by $\Tr_{q^n/q^m}(x)$ we will denote the trace of $x$ over $\F_{q^m}$. If $n$ is clear from the context and $m=1$ then we will simply write $\N(x)$ and $\Tr(x)$.}

\begin{corollary}\label{norm}
If the kernel of a $q^s$-polynomial $f(x)=a_0x+a_1x^{q^s}+\cdots+a_{k-1}x^{q^{s(k-1)}}-x^{q^{sk}}$ has dimension $k$, then $\N(a_0)=(-1)^{n(k+1)}$.
\end{corollary}

\begin{corollary}\label{fix}
Let $A$ be a matrix as in Theorem \ref{main}. The condition
\[A A^{q^s} \cdots A^{q^{s(n-1)}}=I_k\]
is satisfied if and only if $A A^{q^s} \cdots A^{q^{s(n-1)}}$ fixes $\mathbf{e}_0=(1,0,\ldots,0)$.
\end{corollary}
\begin{proof}
The only if part is trivial, we prove the if part by induction on $0 \leq i \leq k-1$.
Suppose $A A^{q^s} \cdots A^{q^{s(n-1)}}\mathbf{e}_i^T=\mathbf{e}_i^T$ for some $0 \leq i \leq k-1$.
Then by taking $q^s$-th powers of each entry we get $A^{q^s} A^{q^{2s}} \cdots A \mathbf{e}_i^T=\mathbf{e}_i^T$. Since $A\mathbf{e}_i^T=\mathbf{e}_{i+1}^T$ this yields
$A^{q^s} A^{q^{2s}} \cdots A^{q^{s(n-1)}} \mathbf{e}_{i+1}^T=\mathbf{e}_{i}^T$. Then multiplying both sides by $A$ yields
$AA^{q^s} A^{q^{2s}} \cdots A^{q^{s(n-1)}} \mathbf{e}_{i+1}^T=\mathbf{e}_{i+1}^T$.
\end{proof}

\medskip

Consider a $q^s$-polynomial $f(x)=a_0x+a_1x^{q^s}+\cdots+a_{k-1}x^{q^{s(k-1)}}-x^{q^{sk}}$, the matrix $A \in \F_{q^n}^{k\times k}$ as in Theorem \ref{main} and the semilinear map $\tau$ defined in \eqref{law}.

Note that

\[ \mathbf{e}_0^\tau=(0,1,0,\ldots,0)=\mathbf{e}_1 \]
\[ \mathbf{e}_0^{\tau^2}=(0,0,1,\ldots,0)=\mathbf{e}_2 \]
\[ \vdots \]
\[ \mathbf{e}_0^{\tau^{k-1}}=(0,0,0,\ldots,1)=\mathbf{e}_{k-1} \]
\[ \mathbf{e}_0^{\tau^k}=(a_0,a_1,a_2,\ldots,a_{k-1}) \]
\begin{equation}\label{1step}
\mathbf{e}_0^{\tau^{k+1}}=(a_0a_{k-1}^{q^s},a_0^{q^s}+a_1a_{k-1}^{q^s},a_1^{q^s}+a_2a_{k-1}^{q^s},\ldots,a_{k-2}^{q^s}+a_{k-1}^{q^s+1}).
\end{equation}
Hence, if
\[ \mathbf{e}_0^{\tau^i}=(Q_{0,i},Q_{1,i},\ldots,Q_{k-1,i}) \]
where $Q_{j,i}$ can be seen as polynomials in $a_0,a_1,\ldots,a_{k-1}$,
for $i\geq 0$, then
\[
\begin{array}{llc}
\mathbf{e}_0^{\tau^{i+1}}=(a_0Q_{k-1,i}^{q^s},Q_{0,i}^{q^s}+a_1Q_{k-1,i}^{q^s}, \ldots, Q_{k-2,i}^{q^s}+a_{k-1}Q_{k-1,i}^{q^s}),
\end{array} \]
i.e. the polynomials $Q_{j,i}$ for $0\leq j \leq k-1$ can be defined by the following recursive relations for $0\leq i \leq k-1$:
%
%
%
%
%
%
\begin{equation*}
Q_{j,i}=\left\{\begin{array}{llc} 1 & \text{if}\, j=i, \\ 0 & \text{otherwise}, \end{array}\right.\\
\end{equation*}
and by the following relations for $i\geq k$:
\begin{equation}\label{1el,elsucc}
\begin{array}{lllc}
Q_{0,i+1}=a_0Q_{k-1,i}^{q^s}\\
Q_{j,i+1}=Q_{j-1,i}^{q^s}+a_jQ_{k-1,i}^{q^s}.
\end{array}
\end{equation}

Now, we are able to prove the following.

\begin{theorem}
\label{rec}
The kernel of a $q^s$-polynomial $f(x)=a_0x+a_1x^{q^s}+\cdots+a_{k-1}x^{q^{s(k-1)}}-x^{q^{sk}}\in \F_{q^n}[x]$, where $\gcd(s,n)=1$, has dimension $k$ if and only if
\begin{equation}\label{necsuff}
Q_{j,n}(a_0,a_1,\ldots,a_{k-1})= \left\{ \begin{array}{llcc}
1 & \text{if }  j=0, \\ 0 & \text{otherwise}. \end{array} \right.
\end{equation}
\end{theorem}
\begin{proof}
Relations \eqref{1el,elsucc} can be written as follows

\[ \left(
\begin{array}{c}
Q_{0,i+1}\\
Q_{1,i+1}\\
\vdots\\
Q_{k-1,i+1}
\end{array}
\right) =
\left(
\begin{array}{cccccccccccc}
0 & 0 & \cdots & 0 & a_0 \\
1 & 0 & \cdots & 0 & a_1 \\
0 & 1 & \cdots & 0 & a_2 \\
\vdots & \vdots & & \vdots & \vdots \\
0 & 0 & \cdots & 1 & a_{k-1}
\end{array}
\right)
\left(
\begin{array}{c}
Q_{0,i}^{q^s}\\
Q_{1,i}^{q^s}\\
\vdots\\
Q_{k-1,i}^{q^s}
\end{array}
\right),
 \]
with $i \in \{0,\ldots,n-1\}$. Also, $(Q_{0,0},Q_{1,0},\ldots,Q_{k-1,0})=(1,0,\ldots,0)$ and $\mathbf{e}_0^{\tau^t}=(Q_{0,t},\ldots,Q_{k-1,t})$ for $t \in \{0,\ldots,n\}$.
By Theorem \ref{main} and by Corollary \ref{fix}, the kernel of $f(x)$ has dimension $k$ if and only if $\mathbf{e}_0=(Q_{0,0},Q_{1,0},\ldots,Q_{k-1,0})$ is fixed by $A A^{q^s} \cdots A^{q^{s(n-1)}}$, so this happens if and only if
\[ \mathbf{e}_0^{\tau^n}=(Q_{0,n},Q_{1,n},\ldots,Q_{k-1,n})=(1,0,\ldots,0). \]
\end{proof}

\medskip

Theorem \ref{rec} with $k=n-1$ and $s=1$ gives the following well-known result as a corollary.

\begin{corollary}\cite[Theorem 2.24]{LN}\label{trace}
The dimension of the kernel of a $q$-polynomial $f(x) \in \F_{q^n}[x]$ is $n-1$ if and only if there exist $\alpha, \beta \in \F_{q^n}^*$ such that
\[f(x)=\alpha\Tr(\beta x).\]
\end{corollary}
%
%
%
%
%
%
%

Again from Theorem \ref{rec} we can deduce the following.

\begin{corollary}\cite[Ex. 2.14]{LN}
	\label{bin}
The $q^s$-polynomial $a_0x-x^{q^{sk}}\in \F_{q^n}[x]$, with $\gcd(s,n)=1$ and $1\leq k\leq n-1$, admits $q^k$ roots if and only if $k \mid n$ and $\N_{q^{n}/q^{k}}(a_0)=1$.
\end{corollary}
%
%
%
%

\subsection{When the $q^s$-degree equals $n-2$}
\label{k=n-2}

In this section we investigate $q^s$-polynomials \[f(x)=a_0x+a_1x^{q^s}+\cdots+a_{n-3}x^{q^{s(n-3)}}-x^{q^{s(n-2)}}\]
with $\gcd(s,n)=1$. By Theorem \ref{rec}, $\dim \ker{f(x)} =n-2$ if and only if $a_0,a_1,\ldots,a_{n-3}$ satisfy the following system of equations

\begin{equation}\label{star}
\left\{ \begin{array}{llllllc}
Q_{0,n}=a_0(a_{n-4}^{q^{2s}}+a_{n-3}^{q^{2s}+q^s})=1,\\
Q_{1,n}=a_0^{q^s} a_{n-3}^{q^{2s}}+a_1(a_{n-4}^{q^{2s}}+a_{n-3}^{q^{2s}+q^s})=0,\\
Q_{2,n}=a_0^{q^{2s}}+a_{n-3}^{q^{2s}}a_1^{q^s}+a_2(a_{n-4}^{q^{2s}}+a_{n-3}^{q^{2s}+q^s})=0,\\
Q_{3,n}=a_1^{q^{2s}}+a_{n-3}^{q^{2s}}a_2^{q^s}+a_3(a_{n-4}^{q^{2s}}+a_{n-3}^{q^{2s}+q^s})=0,\\
\vdots \\
Q_{n-3,n}=a_{n-5}^{q^{2s}}+a_{n-3}^{q^{2s}}a_{n-4}^{q^s}+a_{n-3}(a_{n-4}^{q^{2s}}+a_{n-3}^{q^{2s}+q^s})=0,
\end{array} \right. \end{equation}
which is equivalent to
\begin{equation}\label{reln-2}
\left\{ \begin{array}{lllc}
a_0(a_{n-4}^{q^{2s}}+a_{n-3}^{q^{2s}+q^s})=1,\\
a_1=-a_0^{q^s+1}a_{n-3}^{q^{2s}}=:g_1(a_0,a_{n-3}),\\
a_j=-a_{j-2}^{q^{2s}}a_0-a_{n-3}^{q^{2s}}a_{j-1}^{q^s}a_0=:g_j(a_0,a_{n-3}), \mbox{ for $2 \leq j \leq n-3$}.
\end{array} \right.
\end{equation}
So, $\dim_{\F_q} \ker{f(x)} =n-2$ if and only if $a_0$ and $a_{n-3}$ satisfy the equations
\[
\left\{ \begin{array}{llc}
a_0(g_{n-4}(a_0,a_{n-3})^{q^{2s}}+a_{n-3}^{q^{2s}+q^s})=1,\\
a_{n-3}=g_{n-3}(a_0,a_{n-3}),
\end{array} \right.
\]
and $a_j=g_j(a_0,a_{n-3})$ for $j \in \{1,\ldots,n-4\}$.

\begin{theorem}\label{newrelt}
Suppose that $f(x)=a_0x+a_1x^q+\cdots+a_{n-3}x^{q^{n-3}}-x^{q^{n-2}}$ has maximum kernel. Then for $t\geq 2$ with $\gcd(t-1,n)=1$ the coefficients $a_{t-2}$ and $a_{n-t}$ are non-zero and, with $s=n-t+1$,
\begin{equation}\label{newrel}
a_{n-2t+1}a_{t-2}^{q^{2s}+q^s}=-a_{n-t}^{q^s+1}a_{2t-3}^{q^{2s}}.
\end{equation}
Also, it holds that
\begin{equation}\label{newrel2}
-a_{n-t}(-a_{t-2}^{q^s}a_{3t-4}^{q^{2s}}+a_{2t-3}^{q^{2s}+q^s})=a_{t-2}^{q^{2s}+q^s+1}.
\end{equation}
In particular, for $t\geq 2$ with $\gcd(t-1,n)=1$ we get
\begin{equation}
\label{ek1}
\N(a_{n-t})=(-1)^n\N(a_{t-2})
\end{equation}
and
\begin{equation}
\label{ek2}
\N(a_{n-2t+1})=(-1)^n\N(a_{2t-3}),
\end{equation}
where $n-2t+1$ and $2t-3$ are considered modulo $n$.
\end{theorem}
\begin{proof}
Let $t \geq 2$ with $\gcd(t-1,n)=1$ and consider the polynomial $F(x)=f(x^{q^t})$, that is,
\[ F(x)=a_0x^{q^t}+a_1x^{q^{t+1}}+\cdots+a_{n-3}x^{q^{n+t-3}}-x^{q^{n+t-2}}.\]
Clearly $\dim_{\F_q} \ker F =\dim_{\F_q} \ker f=n-2$.
By renaming the coefficients, $F(x)$ can be written as
\[ F(x)=\alpha_0x+\alpha_1x^{q^{n-t+1}}+\alpha_2x^{q^{2(n-t+1)}}+\cdots+\alpha_{n-3}x^{q^{(n-t+1)(n-3)}}+\alpha_{n-2}x^{q^{(n-t+1)(n-2)}}\]
\[ =\alpha_0x+\alpha_1x^{q^{n-t+1}}+\cdots+\alpha_{n-3}x^{q^{3t-3}}+\alpha_{n-2}x^{q^{2t-2}}. \]
Since $F(x)$ has maximum kernel, by the second equation of \eqref{reln-2} we get $\alpha_0 \neq 0$, $\alpha_{n-2}\neq 0$ and the following relation
\begin{equation}\label{reln-21}
-\frac{\alpha_1}{\alpha_{n-2}}=-\left(-\frac{\alpha_0}{\alpha_{n-2}}\right)^{q^s+1}\left(-\frac{\alpha_{n-3}}{\alpha_{n-2}}\right)^{q^{2s}}.
\end{equation}
The coefficient $\alpha_j$ of $F(x)$ equals the coefficient $a_i$ of $f(x)$ with $i\equiv n-t+j(1-t) \pmod{n}$, in particular
\begin{equation}\label{corr}
\left\{ \begin{array}{lllllc} \alpha_0 = a_{n-t}, \\ \alpha_1=a_{n-2t+1},\\ \alpha_{n-3}=a_{2t-3},\\ \alpha_{n-2}=a_{t-2}, \\ \alpha_{n-4}=a_{3t-4}, \end{array} \right.
\end{equation}
and by \eqref{reln-21}, we get that $a_{t-2}$ and $a_{n-t}$ are nonzero, and
\[a_{n-2t+1}a_{t-2}^{q^{2s}+q^s}=-a_{n-t}^{q^s+1}a_{2t-3}^{q^{2s}},\]
which gives \eqref{newrel}.
\smallskip
The first equation of \eqref{reln-2} gives
\[-\frac{\alpha_0}{\alpha_{n-2}} \left( \left(-\frac{\alpha_{n-4}}{\alpha_{n-2}}\right)^{q^{2s}}+\left( - \frac{\alpha_{n-3}}{\alpha_{n-2}} \right)^{q^{2s}+q^s} \right)=1,\]
that is,
\[ -\alpha_0(-\alpha_{n-2}^{q^s}\alpha_{n-4}^{q^{2s}}+\alpha_{n-3}^{q^{2s}+q^s})=\alpha_{n-2}^{q^{2s}+q^s+1}. \]
Then \eqref{corr} and $\alpha_{n-4}=a_{3t-4}$ imply
\[-a_{n-t}(-a_{t-2}^{q^s}a_{3t-4}^{q^{2s}}+a_{2t-3}^{q^{2s}+q^s})=a_{t-2}^{q^{2s}+q^s+1}, \]
which gives \eqref{newrel2}.
\smallskip
By Corollary \ref{norm} with $s=n-t+1$ we obtain
\[ \N\left(-\frac{\alpha_0}{\alpha_{n-2}}\right)=1, \]
and taking \eqref{corr} into account we get
\[ \N(a_{n-t})=(-1)^n\N(a_{t-2}). \]
Then \eqref{newrel} and the previous relation yield
\[\N(a_{n-2t+1})=(-1)^n\N(a_{2t-3}).\]
\end{proof}

\begin{proposition}\label{trace2}
Let $f(x)$ be a $q^s$-polynomial with $q^s$-degree $n-2$ and with maximum kernel. If the coefficient of $x^{q^s}$ is zero, then $n$ is even and
$f(x)=\alpha \Tr_{q^n/q^2}(\beta x)$ for some $\alpha,\beta \in \F_{q^n}^*$.
\end{proposition}
\begin{proof}
We may assume $f(x)=a_0x+a_1x^{q^s}+\cdots+a_{n-3}x^{q^{s(n-3)}}-x^{q^{s(n-2)}}$ with $a_1=0$. By the second equation of \eqref{reln-2}, it follows that $a_{n-3}=0$.
By the third equation of \eqref{reln-2}, we get that $a_j=0$ for every odd integer $j\in\{3,\ldots,n-3\}$.
If $j$ is even then we have
\begin{equation}\label{tr2}
a_j=(-1)^{\frac{j}{2}}a_0^{q^{sj}+q^{s(j-2)}+\cdots+q^{2s}+1}.
\end{equation}
If $n-3$ is even, then this gives us a contradiction with $j=n-3$.
It follows that $n-3$ is odd and hence $n$ is even. By $\N(a_0)=(-1)^n$, there exists $\lambda \in \F_{q^n}^*$ such that $a_0=-\lambda^{1-q^{s(n-2)}}$.
So, by \eqref{tr2} we get $a_j=\lambda^{q^{js}-q^{s(n-2)}}$,
and hence
\[ f(x)=\frac{\Tr_{q^n/q^2}(\lambda x)}{\lambda^{q^{s(n-2)}}}. \]
\end{proof}

In the next sections we list all the $q^s$-polynomials of $\F_{q^n}$ with maximum kernel for $n\leq 6$. By Corollaries \ref{trace} and \ref{bin} the $n\leq 3$ case can be easily described hence we will consider only the $n\in\{4,5,6\}$ cases.

\medskip

For $f(x)=\sum_{i=0}^{n-1}a_i x^{q^i}\in \tilde\cL_{n,q}$ we denote by
$\hat f(x):=\sum_{i=0}^{n-1}a_i^{q^{n-i}} x^{q^{n-i}}$ the
adjoint (w.r.t. the symmetric non-degenerate bilinear form defined by $\la x, y\ra=\Tr(xy)$) of $f$.

\medskip
By \cite[Lemma 2.6]{BGMP2015}, see also \cite[pages 407--408]{CSMP2016}, the kernel of $f$ and $\hat f$ has the same dimension and hence the following result holds.
\begin{proposition}
\label{adjoint}
If $f(x)\in \tilde\cL_{n,q}$ is a $q^s$-polynomial with maximum kernel, then $\hat f(x)$ is a $q^{n-s}$-polynomial with maximum kernel.
\end{proposition}
This will allow us to consider only the $s\leq n/2$ case.

\subsection{The $n=4$ case}\label{n=4}

In this section we determine the linearized polynomials over $\F_{q^4}$ with maximum kernel. Without loss of generality, we can suppose that the leading coefficient of the polynomial is $-1$.

\smallskip

Because of Proposition \ref{adjoint}, we can assume $s=1$. Corollaries \ref{trace} and \ref{bin} cover the cases when the $q$-degree of $f$ is $1$ or $3$ so from now on we suppose $f(x)=a_0x+a_1x^q-x^{q^2}$.
If $a_1=0$ then we can use again Corollary \ref{bin} and we get
$a_0x-x^{q^2}$, with $\N_{q^4/q^2}(a_0)=1$.
Suppose $a_1\neq 0$. By Equation \eqref{reln-2}, we get the conditions
\[ \left\{ \begin{array}{llc} a_0(a_0^{q^2}+a_1^{q^2+q})=1,\\ a_1=-a_0^{q+1}a_1^{q^2}, \end{array} \right.\]
which is equivalent to
\[ \left\{ \begin{array}{llc} \N_{q^4/q}(a_0)=1, \\ a_1^{q+1}=a_0^{q^2+q+1}-a_0^q, \end{array} \right. \]
see (A1) of Section \ref{Appendix}.

\medskip

Here we list the $q$-polynomials of $\cL_{4,q}$ with maximum kernel, up to a non-zero scalar in $\F_{q^4}^*$. Applying the adjoint operation we can obtain the list of $q^3$-polynomials over $\F_{q^4}$ with maximum kernel. In the following table the $q$-degree will be denoted by $k$.

\begin{table}[htp]
\[
\begin{array}{|c|c|c|c| }
\hline
k & \text{polynomial form} & \text{conditions} \\
\hline
3 & \Tr(\lambda x) & \lambda \in \F_{q^4}^* \\ \hline
2 & a_0x-x^{q^{2}} & \N_{q^4/q^2}(a_0)=1 \\ \hline
2 & a_0x+a_1x^{q}-x^{q^{2}} & \left\{ \begin{array}{llc} \N_{q^4/q}(a_0)=1 \\ a_1^{q+1}=a_0^{q^{2}+q+1}-a_0^{q} \end{array} \right. \\ \hline
1 & a_0x-x^{q} & \N_{q^4/q}(a_0)=1 \\ \hline
\end{array}
\]
\caption{Linearized polynomials of $\F_{q^4}$ with maximum kernel with $s=1$}
\label{n=4t}
\end{table}

\subsection{The $n=5$ case}\label{n=5}

In this section we determine the linearized polynomials over $\F_{q^5}$ with maximum kernel. Without loss of generality, we can suppose that the leading coefficient of the polynomial is $-1$.
Because of Proposition \ref{adjoint}, we can assume $s\in\{1,2\}$. Corollaries \ref{trace} and \ref{bin} cover the cases when the $q^s$-degree of $f$ is $1$ or $4$.
\smallskip
First we suppose that $f$ has $q^s$-degree $3$, i.e.
\[ f(x)=a_0x+a_1x^{q^s}+a_2x^{q^{2s}}-x^{q^{3s}}. \]
From \eqref{reln-2}, $f(x)$ has maximum kernel if and only if $a_0$, $a_1$ and $a_2$ satisfy the following system:
\[\left\{ \begin{array}{lllc}
a_1=-a_0^{q^s+1}a_2^{q^{2s}},\\
-a_0^{q^{3s}+q^{2s}+1}a_2^{q^{4s}}+a_2^{q^{2s}+q^s}a_0=1,\\
a_2=-a_0^{q^{2s}+1}+a_2^{q^{3s}+q^{2s}}a_0^{q^{2s}+q^s+1},
\end{array} \right. \]

\noindent which is equivalent to
\[\left\{ \begin{array}{lllc}
\N(a_0)=1,\\
a_1=-a_0^{q^s+1}a_2^{q^{2s}},\\
-a_0^{q^{3s}+q^{2s}+1}a_2^{q^{4s}}+a_0a_2^{q^{2s}+q^s}=1,
\end{array} \right.\]

\noindent see (A2) of Section \ref{Appendix}.

\smallskip

\noindent Suppose now that the $q^s$-degree is $2$, i.e.
\[f(x)=a_0x+a_1x^{q^s}-x^{q^{2s}}.\]
By  Theorem \ref{rec} the polynomial $f(x)$ has maximum kernel if and only if its coefficients satisfy
\[\left\{ \begin{array}{llc} a_0(a_0^{q^{2s}}a_1^{q^{3s}}+a_1^{q^s}(a_0^{q^{3s}}+a_1^{q^{3s}+q^{2s}}))=1, \\ a_0^{q^s+1}(a_0^{q^{3s}}+a_1^{q^{3s}+q^{2s}})+a_1=0, \end{array} \right. \]
which is equivalent to
\[\left\{ \begin{array}{llc} \N(a_0)=-1, \\a_0^{q^s}+a_1^{q^s+1}=a_0^{q^{2s}+q^s+1}a_1^{q^{3s}}, \end{array} \right.\]
see (A3) of Section \ref{Appendix}.

\medskip

Here we list the $q^s$-polynomials, $s\in\{1,2\}$ of $\cL_{5,q}$ with maximum kernel, up to a non-zero scalar in $\F_{q^5}^*$. Applying the adjoint operation we can obtain the list of $q^t$-polynomials, $t\in\{3,4\}$, over $\F_{q^5}$ with maximum kernel. As before, the $q^s$-degree is denoted by $k$.

\begin{table}[htp]
\[
\begin{array}{ |c|c|c|c| }

\hline
$k$ & \text{polynomial form} & \text{conditions} \\
\hline
4 &  \Tr(\lambda x) & \lambda \in \F_{q^5}^* \\ \hline
3 & a_0x+a_1x^{q^s}+a_2x^{q^{2s}}-x^{q^{3s}} & \left\{ \begin{array}{lllc}
\N(a_0)=1\\
a_1=-a_0^{q^s+1}a_2^{q^{2s}}\\
-a_0^{q^{3s}+q^{2s}+1}a_2^{q^{4s}}+a_0a_2^{q^{2s}+q^s}=1
\end{array} \right. \\ \hline
2 & a_0x+a_1x^{q^s}-x^{q^{2s}} & \left\{ \begin{array}{llc} \N(a_0)=-1 \\ a_1^{q^s+1}+a_0^{q^s}=a_1^{q^{3s}}a_0^{q^{2s}+q^s+1} \end{array} \right. \\ \hline
1 & a_0x-x^{q^s} & \N(a_0)=1 \\ \hline
\end{array}
\]
\caption{Linearized polynomials of $\F_{q^5}$ with maximum kernel with $s\in \{1,2\}$}
\label{n=5t}
\end{table}

\subsection{The $n=6$ case}\label{n=6}

In this section we determine the linearized polynomials over $\F_{q^6}$ with maximum kernel. Without loss of generality, we can suppose that the leading coefficient of the polynomial is $-1$.
Because of Proposition \ref{adjoint}, we can assume $s=1$. Corollaries \ref{trace} and \ref{bin} cover the cases when the $q$-degree of $f$ is $1$ or $5$. As before, denote by $k$ the $q^s$-degree of $f$.
\smallskip

\smallskip

We first consider the case $k=2$, i.e. $f(x)=a_0x+a_1x^{q^s}-x^{q^{2s}}$. By Theorem \ref{rec}, $f(x)$ has maximum kernel if and only if the coefficients satisfy

\[ \left\{ \begin{array}{lllc} \N(a_0)=1, \\ (a_0^{q}+a_1^{q+1})^{q^{3}}=a_0^{q^{5}+q^{4}+q^{3}}(a_0^{q}+a_1^{q+1}), \\
a_1^{q^{4}}a_0^{q^{3}}+a_1^{q^{2}}(a_0^{q^{4}}+a_1^{q^{4}+q^{3}})=-\frac{a_1}{a_0^{q+1}},\end{array} \right. \]

\noindent see (A4) of Section \ref{Appendix}.

\smallskip

If $k=3$, then $f(x)=a_0x+a_1x^{q^s}+a_2x^{q^{2s}}-x^{q^{3s}}$, and by Theorem \ref{rec} it has maximum kernel if and only the coefficients fulfill

\[\left\{ \begin{array}{llllc}
\N(a_0)=1,\\
a_0^{q^{3}+q+1}+a_2^{q^{3}}a_1^{q^{2}}a_0^{q+1}-a_2^{q}a_1=a_0^{q},\\
a_2^{q+1}= -a_0^{q^{3}+q^{2}+q+1}a_1^{q^{4}}-a_1^{q},\\
a_1^{q+1}= a_2a_0^{q}+a_0^{q^{2}+q+1}a_2^{q^{3}},
\end{array}\right.\]

\noindent see (A5) of Section \ref{Appendix}. Note that $a_1=0$ if and only if $a_2=0$ and in this case we get the trace over $\F_{q^3}$.

\smallskip

Finally, let $k=4$. Then the polynomial $f(x)=a_0x+a_1x^{q^s}+a_2x^{q^{2s}}+a_3x^{q^{3s}}-x^{q^{4s}}$ has maximum kernel if and only if the coefficients satisfy

\[ \left\{ \begin{array}{lllllc}
\N(a_0)=1,\\
a_0(-a_0^{q^{4}+q^{2}}+a_3^{q^{5}+q^{4}}a_0^{q^{4}+q^{3}+q^{2}}+a_3^{q^{2}+q})=1,\\
a_1=-a_0^{q+1}a_3^{q^{2}},\\
a_2=-a_0^{q^{2}+1}+a_3^{q^{3}+q^{2}}a_0^{q^{2}+q+1},\\
a_3=a_3^{q^{4}}a_0^{q^{3}+q^{2}+1}+a_3^{q^{2}}a_0^{q^{3}+q+1}-a_0^{q^{3}+q^{2}+q+1}a_3^{q^{4}+q^{3}+q^{2}},
\end{array} \right.\]

\noindent see (A6) of Section \ref{Appendix}.

\medskip

Here we list the $q$-polynomials of $\cL_{6,q}$ with maximum kernel, up to a non-zero scalar in $\F_{q^6}^*$. Applying the adjoint operation we can obtain the list of $q^5$-polynomials over $\F_{q^6}$ with maximum kernel.

\begin{small}
\begin{sidewaystable}
\centering
\caption{Linearized polynomials of $\F_{q^6}$ with maximum kernel with $s=1$}
\begin{tabular}{ |c|c|c| }
\hline
$k$ & \text{polynomial form} & \text{conditions} \\
\hline
$5$ & $\Tr_{q^6/q}(\lambda x)$ & $\lambda \in \F_{q^6}^*$ \\
\hline
$4$ & $a_0x+a_1x^{q}+a_2x^{q^{2}}+a_3x^{q^{3}}-x^{q^{4}}$ &
$\left\{ \begin{array}{llllllc}
a_1\neq 0\\
\N(a_0)=1\\
a_0(-a_0^{q^{4}+q^{2}}+a_3^{q^{5}+q^{4}}a_0^{q^{4}+q^{3}+q^{2}}+a_3^{q^{2}+q})=1\\
a_1=-a_0^{q+1}a_3^{q^{2}}\\
a_2=-a_0^{q^{2}+1}+a_3^{q^{3}+q^{2}}a_0^{q^{2}+q+1}\\
a_3=a_3^{q^{4}}a_0^{q^{3}+q^{2}+1}+a_3^{q^{2}}a_0^{q^{3}+q+1}-a_0^{q^{3}+q^{2}+q+1}a_3^{q^{4}+q^{3}+q^{2}}
\end{array} \right. $ \\
\hline
$4$ & $\Tr_{q^6/q^2}(\lambda x)$ & $\lambda \in \F_{q^6}^*$ \\
\hline
$3$ & $a_0x+a_1x^{q}+a_2x^{q^{2}}-x^{q^{3}}$ & $\left\{ \begin{array}{llllc}
\N(a_0)=1\\
a_0^{q^{3}+q+1}+a_2^{q^{3}}a_1^{q^{2}}a_0^{q+1}-a_2^{q}a_1=a_0^{q}\\
a_2^{q+1}= -a_0^{q^{3}+q^{2}+q+1}a_1^{q^{4}}-a_1^{q}\\
a_1^{q+1}= a_2a_0^{q}+a_0^{q^{2}+q+1}a_2^{q^{3}}
\end{array}\right.$ \\
\hline
$3$ & $\Tr_{q^6/q^3}(\lambda x)$ & $\lambda \in \F_{q^6}^*$ \\
\hline
$2$ & $a_0x+a_1x^{q}-x^{q^{2}}$ & $\left\{ \begin{array}{llllc} a_1 \neq 0\\
\N(a_0)=1 \\
(a_0^{q}+a_1^{q+1})^{q^{3}}=a_0^{q^{5}+q^{4}+q^{3}}(a_0^{q}+a_1^{q+1}) \\
a_1^{q^{4}}a_0^{q^{3}}+a_1^{q^{2}}(a_0^{q^{4}}+a_1^{q^{4}+q^{3}})=-\frac{a_1}{a_0^{q+1}}\end{array} \right.$\\
\hline
$2$ & $a_0x-x^{q^{2}}$ & $\N_{q^6/q^2}(a_0)=1$ \\ \hline
$1$ & $a_0x-x^{q}$ & $\N_{q^6/q}(a_0)=1$ \\ \hline
\end{tabular}\label{n=6t}
\end{sidewaystable}
\end{small}

\newpage

\section{Application}

As an application of Theorem \ref{main} we are able to prove the following result on the splitting field of $q$-polynomials.

\begin{theorem}
\label{appl}
Let $f(x)=a_0x+a_1x^{q}+\cdots+a_{k-1}x^{q^{k-1}}-x^{q^{k}}\in \F_{q^n}[x]$ with $a_0\neq 0$ and let $A$ be defined as in \eqref{matrix}.
Then the splitting field of $f(x)$ is $\F_{q^{nm}}$ where $m$ is the (multiplicative) order of the matrix $B:=A A^{q}\cdots A^{q^{n-1}}$.
\end{theorem}
\begin{proof}
The derivative of $f(x)$ is non-zero and hence $f(x)$ has $q^k$ distinct roots in some algebraic extension of $\F_{q^n}$.
Suppose that $\F_{q^{nm}}$ is the splitting field of $f(x)$ and let $t$ denote the order of $B$.
Then the kernel of the $\F_q$-linear $\F_{q^{nm}}\rightarrow \F_{q^{nm}}$ map defined as $x \mapsto f(x)$ has dimension $k$ over $\F_q$ and hence by Theorem \ref{main} we have
\[A A^{q}\cdots A^{q^{nm-1}}=I_k.\]
Since the coefficients of $A$ are in $\F_{q^n}$, this is equivalent to $B^m=I_k$ and hence $t \mid m$.
On the other hand
\[B^t=A A^{q}\cdots A^{q^{nt-1}}=I_k\]
and hence again by Theorem \ref{main} the kernel of the $\F_q$-linear
$\F_{q^{nt}}\rightarrow \F_{q^{nt}}$ map defined as $x \mapsto f(x)$ has dimension $k$ over $\F_q$. It follows that $\F_{q^{nm}}$ is a subfield of $\F_{q^{nt}}$ from which $m \mid t$.
\end{proof}

A further application of Theorem \ref{main} is the following.

\begin{theorem}
Let $n,m,s$ and $t$ be positive integers such that $\gcd(s,nm)=\gcd(t,nm)=1$ and $s \equiv t \pmod{m}$. Let  $f(x)=a_0x+a_1x^{q^s}+\cdots+a_{k-1}x^{q^{s(k-1)}}-x^{q^{sk}}$ and $g(x)=a_0x+a_1x^{q^t}+\cdots+a_{k-1}x^{q^{t(k-1)}}-x^{q^{tk}}$, where $a_0,a_1,\ldots,a_{k-1}\in\F_{q^m}$.
The kernel of $f(x)$ considered as a linear transformation of $\F_{q^{nm}}$ has dimension $k$ if and only if the kernel of $g(x)$ considered as a linear transformation of $\F_{q^{nm}}$ has dimension $k$.
\end{theorem}
\begin{proof}
Denote by $A$ the matrix associated with $f(x)$ as in \eqref{matrix}. By hypothesis, $A \in \F_{q^m}^{k\times k}$ and it is the same as the matrix associated with $g(x)$.
By Theorem \ref{main} the kernel of $f(x)$, considered as a linear transformation of $\F_{q^{nm}}$, has dimension $k$ if and only if
\[AA^{q^s}\cdots A^{q^{s(nm-1)}}=I_k.\]
Since $s \equiv t \pmod{m}$, we have
\[AA^{q^s}\cdots A^{q^{s(nm-1)}}=AA^{q^t}\cdots A^{q^{t(nm-1)}}=I_k, \]
and, again by Theorem \ref{main}, this holds if and only if the kernel of $g(x)$, considered as a linear transformation of $\F_{q^{nm}}$, has dimension $k$.
\end{proof}

\subsection*{Addendum}

During the ``Combinatorics 2018" conference, the fourth author presented the results of this paper in the talk entitled ``On $q$-polynomials with maximum kernel". In the same conference John Sheekey  presented a joint work with Gary McGuire \cite{McGS} in his talk entitled ``Ranks of Linearized Polynomials and Roots of Projective Polynomials". It turned out that, independently from the authors of the present paper, they also obtained similar results.

\newpage

\section{Appendix}\label{Appendix}

In this section we  develop some calculations regarding the relations on the coefficients of a linearized polynomials with maximum kernel presented in Sections \ref{n=4}, \ref{n=5} and \ref{n=6}, see also \cite{PhDthesisFerd}.

\begin{itemize}
  \item [(A1)] By Equation \eqref{star} with $n=4$, $s=1$ and $k=2$, we get the conditions
        \[ \Sigma \colon \left\{ \begin{array}{llc} a_0(a_0^{q^{2}}+a_1^{q^{2}+q})=1,\\ a_1=-a_0^{q+1}a_1^{q^{2}}. \end{array} \right.\]
        By Corollary \ref{norm}, the system $\Sigma$ is equivalent to the following system
        \[ \Sigma' \colon \left\{ \begin{array}{lllc} \N_{q^4/q}(a_0)=1,\\ a_0(a_0^{q^{2}}+a_1^{q^{2}+q})=1,\\ a_1=-a_0^{q+1}a_1^{q^{2}}, \end{array} \right. \]
        which can be rewritten as follows
        \[ \Sigma' \colon \left\{ \begin{array}{lllc} \N_{q^4/q}(a_0)=1,\\ a_1^{q^{2}-1}=-\frac{1}{a_0^{q+1}},\\ a_1^{q+1}=a_0^{q^{2}+q+1}-a_0^{q}. \end{array} \right. \]
        Now consider the system
        \[\Sigma^*\colon \left\{ \begin{array}{llc} \N_{q^4/q}(a_0)=1, \\ a_1^{q+1}=a_0^{q^{2}+q+1}-a_0^{q}. \end{array} \right. \]
        Clearly, $S(\Sigma') \subseteq S(\Sigma^*)$, where $S(\Sigma')$ and $S(\Sigma^*)$ denote the set of solutions of $\Sigma'$ and $\Sigma^*$, respectively.
        Let $(a_0,a_1) \in S(\Sigma^*)$, then by using the norm condition on $a_0$
        \[ a_1^{q^{2}-1}=\left( \frac{1}{a_0^{q^{3}}} -a_0^{q} \right)^{q-1}=\left(\frac{1-a_0^{q+q^{3}}}{a_0^{q^{3}}}\right)^{q-1}=\]
        \[=\frac{1-a_0^{1+q^{2}}}{1-a_0^{q+q^{3}}}a_0^{q^{3}-1}=\frac{1-\frac{1}{a_0^{q+q^{3}}}}{1-a_0^{q+q^{3}}}a_0^{q^{3}-1}=-\frac{1}{a_0^{q+1}}, \]
        i.e. $(a_0,a_1) \in S(\Sigma')$ and hence $S(\Sigma^*)=S(\Sigma')=S(\Sigma)$.
  \item [(A2)] From \eqref{star} with $n=5$, $\gcd(s,5)=1$ and $k=3$, we get the following conditions:
        \[\Sigma\colon \left\{ \begin{array}{lllc}
        a_0(a_1^{q^{2s}}+a_2^{q^{2s}+q^s})=1,\\
        a_1=-a_0^{q^s+1}a_2^{q^{2s}},\\
        a_2=-a_0^{q^{2s}+1}-a_2^{q^{2s}}a_1^{q^s}a_0.
        \end{array} \right. \]
        By Corollary \ref{norm}, $\Sigma$ is equivalent to
        \[ \Sigma'\colon \left\{ \begin{array}{llllc}
        \N_{q^5/q}(a_0)=1,\\
        a_0(a_1^{q^{2s}}+a_2^{q^{2s}+q^s})=1,\\
        a_1=-a_0^{q^s+1}a_2^{q^{2s}},\\
        a_2=-a_0^{q^{2s}+1}-a_2^{q^{2s}}a_1^{q^s}a_0.
        \end{array} \right.  \]
        which can be rewritten as follows
        \[ \Sigma'\colon \left\{ \begin{array}{llllc}
        \N_{q^5/q}(a_0)=1,\\
        a_1=-a_0^{q^s+1}a_2^{q^{2s}},\\
        -a_0^{q^{3s}+q^{2s}+1}a_2^{q^{4s}}+a_2^{q^{2s}+q^s}a_0=1,\\
        a_2=-a_0^{q^{2s}+1}+a_2^{q^{3s}+q^{2s}}a_0^{q^{2s}+q^s+1}.
        \end{array} \right.  \]
        By raising the third equation to $q^s$ and multiplying by $a_0^{q^{2s}+1}$, since $\N(a_0)=1$, we get the fourth equation.
        Therefore $\Sigma'$, and hence $\Sigma$, is equivalent to
        \[\left\{ \begin{array}{lllc}
        \N_{q^5/q}(a_0)=1,\\
        a_1=-a_0^{q^s+1}a_2^{q^{2s}},\\
        -a_0^{q^{3s}+q^{2s}+1}a_2^{q^{4s}}+a_0a_2^{q^{2s}+q^s}=1.
        \end{array} \right.\]
  \item [(A3)] Applying Theorem \ref{rec} with $n=5$ and $k=2$, we get that the polynomial $f(x)$ has maximum kernel if and only if its coefficients satisfy
        \[ \Sigma \colon \left\{ \begin{array}{llc} Q_{0,5}=a_0(a_0^{q^{2s}}a_1^{q^{3s}}+a_1^{q^s}(a_0^{q^{3s}}+a_1^{q^{3s}+q^{2s}}))=1 ,\\ Q_{1,5}=a_0^{q^s}(a_0^{q^{3s}}+a_1^{q^{3s}+q^{2s}})+a_1(a_0^{q^{2s}}a_1^{q^{3s}}+a_1^{q^s} (a_0^{q^{3s}}+a_1^{q^{3s}+q^{2s}}))=0, \end{array} \right. \]
        which is equivalent to
        \[\left\{ \begin{array}{lllc} \N_{q^5/q}(a_0)=-1,\\ a_0(a_0^{q^{2s}}a_1^{q^{3s}}+a_1^{q^s}(a_0^{q^{3s}}+a_1^{q^{3s}+q^{2s}}))=1, \\ a_0^{q^s+1}(a_0^{q^{3s}}+a_1^{q^{3s}+q^{2s}})+a_1=0, \end{array} \right. \]
        because of Corollary \ref{norm} and since $\displaystyle a_0^{q^{2s}}a_1^{q^{3s}}+a_1^{q^s}(a_0^{q^{3s}}+a_1^{q^{3s}+q^{2s}})=\frac{1}{a_0}$.
        The above system can be rewritten as follows
        \[\left\{ \begin{array}{lllc} \N_{q^5/q}(a_0)=-1,\\ a_0^{q^{3s}}+a_1^{q^{3s}+q^{2s}}=-\frac{a_1}{a_0^{q^s+1}},\\ a_1^{q^{3s}}a_0^{q^{2s}+q^s+1}-a_1^{q^s+1}=a_0^{q^s}, \end{array} \right. \]
        which is equivalent to
        \[\left\{ \begin{array}{lllc} \N_{q^5/q}(a_0)=-1,\\ a_1^{q^s+1}+a_0^{q^s}=a_1^{q^{3s}}a_0^{q^{2s}+q^s+1},\\ a_1a_0^{q^{4s}+q^{3s}+q^{2s}}=-\frac{a_1}{a_0^{q^s+1}}. \end{array} \right. \]
        If the first and the second equations are satisfied, clearly also the last one is fulfilled, hence $\Sigma$ is equivalent to the following system
        \[\left\{ \begin{array}{llc} \N_{q^5/q}(a_0)=-1, \\ a_1^{q^s+1}+a_0^{q^s}=a_1^{q^{3s}}a_0^{q^{2s}+q^s+1}. \end{array} \right.\]
  \item [(A4)] By Theorem \ref{rec}, with $n=6$, $s=1$ and $k=2$, we get
        \begin{small}
        \[ \left\{ \begin{array}{llc} Q_{0,6}=a_0(a_0^{q^{2}}(a_0^{q^{4}}+a_1^{q^{4}+q^{3}})+a_1^{q}(a_0^{q^{3}}a_1^{q^{4}}+a_1^{q^{2}}(a_0^{q^{4}}+a_1^{q^{4}+q^{3}})))=1, \\
        Q_{1,6}=a_0^{q^{2}}a_1(a_0^{q^{4}}+a_1^{q^{4}+q^{3}})+(a_1^{q+1}+a_0^{q})(a_0^{q^{3}}a_1^{q^{4}}+a_1^{q^{2}}(a_0^{q^{4}}+a_1^{q^{4}+q^{3}}))=0,
        \end{array} \right. \]
        \end{small}
        which is equivalent to
        \[ \left\{ \begin{array}{llc} a_0^{q^{2}}(a_0^{q^{4}}+a_1^{q^{4}+q^{3}})+a_1^{q}(a_0^{q^{3}}a_1^{q^{4}}+a_1^{q^{2}}(a_0^{q^{4}}+a_1^{q^{4}+q^{3}}))=\frac{1}{a_0}, \\
        \frac{a_1}{a_0}+a_0^{q}(a_0^{q^{3}}a_1^{q^{4}}+a_1^{q^{2}}(a_0^{q^{4}}+a_1^{q^{4}+q^{3}}))=0,\end{array} \right. \]
        i.e.
        \[ \left\{ \begin{array}{llc} a_0^{q^{2}}(a_0^{q^{4}}+a_1^{q^{4}+q^{3}})+a_1^{q}(a_0^{q^{3}}a_1^{q^{4}}+a_1^{q^{2}}(a_0^{q^{4}}+a_1^{q^{4}+q^{3}}))=\frac{1}{a_0}, \\
        a_0^{q^{3}}a_1^{q^{4}}+a_1^{q^{2}}(a_0^{q^{4}}+a_1^{q^{4}+q^{3}})=-\frac{a_1}{a_0^{q+1}}.\end{array} \right. \]
        By Corollary \ref{norm}, the previous system is equivalent to
        \[ \left\{ \begin{array}{lllc} \N(a_0)=1, \\ a_0^{q^{2}}(a_0^{q^{4}}+a_1^{q^{4}+q^{3}})+a_1^{q}(a_0^{q^{3}}a_1^{q^{4}}+a_1^{q^{2}}(a_0^{q^{4}}+a_1^{q^{4}+q^{3}}))=\frac{1}{a_0}, \\
        a_0^{q^{3}}a_1^{q^{4}}+a_1^{q^{2}}(a_0^{q^{4}}+a_1^{q^{4}+q^{3}})=-\frac{a_1}{a_0^{q+1}},\end{array} \right. \]
        which is equivalent to
        \[ \left\{ \begin{array}{lllc} \N(a_0)=1, \\ a_0^{q^{2}}(a_0^{q}+a_1^{q+1})^{q^{3}}-\frac{a_1^{q+1}}{a_0^{q+1}}=\frac{1}{a_0}, \\
        a_0^{q^{3}}a_1^{q^{4}}+a_1^{q^{2}}(a_0^{q^{4}}+a_1^{q^{4}+q^{3}})=-\frac{a_1}{a_0^{q+1}},\end{array} \right. \]
        hence it is equivalent to
        \[ \left\{ \begin{array}{lllc} \N(a_0)=1, \\ (a_0^{q}+a_1^{q+1})^{q^{3}}=a_0^{q^{5}+q^{4}+q^{3}}(a_0^{q}+a_1^{q+1}), \\
        a_1^{q^{4}}a_0^{q^{3}}+a_1^{q^{2}}(a_0^{q^{4}}+a_1^{q^{4}+q^{3}})=-\frac{a_1}{a_0^{q+1}}.\end{array} \right. \]
  \item [(A5)] By Theorem \ref{rec} with $n=6$, $s=1$ and $k=3$, we get
        \[ \left\{ \begin{array}{lllc}
        a_0Q_{2,5}^{q}=1,\\
        Q_{0,5}^{q}+a_1Q_{2,5}^{q}=0,\\
        Q_{1,5}^{q}+a_2Q_{2,5}^{q}=0,
        \end{array}\right.\]
        where
        \[ \begin{array}{lllc}
        Q_{0,5}=a_0(a_1^{q^{2}}+a_2^{q^{2}+q}),\\
        Q_{1,5}=a_0^{q}a_2^{q^{2}}+a_1(a_1^{q^{2}}+a_2^{q^{2}+q}),\\
        Q_{2,5}=a_0^{q^{2}}+a_2^{q^{2}}a_1^{q}+a_2(a_1^{q^{2}}+a_2^{q^{2}+q}),
        \end{array}\]
        hence we obtain the following system
        \[\left\{ \begin{array}{lllc}
        a_0 (a_0^{q^{3}}+a_2^{q^{3}}a_1^{q^{2}}+a_2^{q}(a_1^{q^{3}}+a_2^{q^{3}+q^{2}}))=1,\\
        \frac{a_1}{a_0}+a_0^{q}(a_1^{q^{3}}+a_2^{q^{3}+q^{2}})=0,\\
        \frac{a_2}{a_0}+a_2^{q^{3}}a_0^{q^{2}}+a_1^{q}(a_1^{q^{3}}+a_2^{q^{3}+q^{2}})=0.
        \end{array}\right.\]
        By Corollary \ref{norm} it is equivalent to
        \[\left\{ \begin{array}{llllc}
        \N(a_0)=1,\\
        a_0 (a_0^{q^{3}}+a_2^{q^{3}}a_1^{q^{2}}+a_2^{q}(a_1^{q^{3}}+a_2^{q^{3}+q^{2}}))=1,\\
        a_1^{q^{3}}+a_2^{q^{3}+q^{2}}=-\frac{a_1}{a_0^{1+q}},\\
        \frac{a_2}{a_0}+a_2^{q^{3}}a_0^{q^{2}}+a_1^{q}(a_1^{q^{3}}+a_2^{q^{3}+q^{2}})=0,
        \end{array}\right.\]
        by substituting the third equation into the others we get
        \[\left\{ \begin{array}{llllc}
        \N(a_0)=1,\\
        a_0 (a_0^{q^{3}}+a_2^{q^{3}}a_1^{q^{2}}-\frac{a_2^{q}a_1}{a_0^{1+q}})=1,\\
        a_1^{q^{3}}+a_2^{q^{3}+q^{2}}=-\frac{a_1}{a_0^{1+q}},\\
        \frac{a_2}{a_0}+a_2^{q^{3}}a_0^{q^{2}}-\frac{a_1^{q+1}}{a_0^{1+q}}=0,
        \end{array}\right.\]
        i.e.
        \[\left\{ \begin{array}{llllc}
        \N(a_0)=1,\\
        a_0^{q^{3}+q+1}+a_2^{q^{3}}a_1^{q^{2}}a_0^{q+1}-a_2^{q}a_1=a_0^{q},\\
        a_2^{q+1}= -a_0^{q^{3}+q^{2}+q+1}a_1^{q^{4}}-a_1^{q},\\
        a_1^{q+1}= a_2a_0^{q}+a_0^{q^{2}+q+1}a_2^{q^{3}}.
        \end{array}\right.\]
  \item [(A6)] Equations \eqref{star} with $n=6$, $s=1$ and $k=4$ are
        \[
        \left\{ \begin{array}{llllc}
        a_0(a_2^{q^{2}}+a_3^{q^{2}+q})=1,\\
        a_0^{q}a_3^{q^{2}}+a_1(a_2^{q^{2}}+a_3^{q^{2}+q})=0,\\
        a_0^{q^{2}}+a_3^{q^{2}}a_1^{q}+a_2(a_2^{q^{2}}+a_3^{q^{2}+q})=0,\\
        a_1^{q^{2}}+a_3^{q^{2}}a_2^{q}+a_3(a_2^{q^{2}}+a_3^{q^{2}+q})=0,
        \end{array} \right.
        \]
        which, by Corollary \ref{norm}, is equivalent to
        \[
        \left\{ \begin{array}{lllllc}
        \N(a_0)=1,\\
        a_0(a_2^{q^{2}}+a_3^{q^{2}+q})=1,\\
        a_0^{q}a_3^{q^{2}}+a_1(a_2^{q^{2}}+a_3^{q^{2}+q})=0,\\
        a_0^{q^{2}}+a_3^{q^{2}}a_1^{q}+a_2(a_2^{q^{2}}+a_3^{q^{2}+q})=0,\\
        a_1^{q^{2}}+a_3^{q^{2}}a_2^{q}+a_3(a_2^{q^{2}}+a_3^{q^{2}+q})=0,
        \end{array} \right.
        \]
        thus it can be rewritten as follows
        \[
        \left\{ \begin{array}{lllllc}
        \N(a_0)=1,\\
        a_2^{q^{2}}+a_3^{q^{2}+q}=\frac{1}{a_0},\\
        a_0^{q}a_3^{q^{2}}+\frac{a_1}{a_0}=0,\\
        a_0^{q^{2}}+a_3^{q^{2}}a_1^{q}+\frac{a_2}{a_0}=0,\\
        a_1^{q^{2}}+a_3^{q^{2}}a_2^{q}+\frac{a_3}{a_0}=0,
        \end{array} \right.
        \]
        and hence
        \[ \left\{ \begin{array}{lllllc}
        \N(a_0)=1,\\
        a_0(a_2^{q^{2}}+a_3^{q^{2}+q})=1,\\
        a_1=-a_0^{q+1}a_3^{q^{2}},\\
        a_2=-a_0^{q^{2}+1}-a_3^{q^{2}}a_1^{q}a_0,\\
        a_3=-a_1^{q^{2}}a_0-a_3^{q^{2}}a_2^{q}a_0,
        \end{array} \right.\]
        i.e.
        \[ \left\{ \begin{array}{lllllc}
        \N(a_0)=1,\\
        a_0(-a_0^{q^{4}+q^{2}}+a_3^{q^{5}+q^{4}}a_0^{q^{4}+q^{3}+q^{2}}+a_3^{q^{2}+q})=1,\\
        a_1=-a_0^{q+1}a_3^{q^{2}},\\
        a_2=-a_0^{q^{2}+1}+a_3^{q^{3}+q^{2}}a_0^{q^{2}+q+1},\\
        a_3=a_3^{q^{4}}a_0^{q^{3}+q^{2}+1}+a_3^{q^{2}}a_0^{q^{3}+q+1}-a_0^{q^{3}+q^{2}+q+1}a_3^{q^{4}+q^{3}+q^{2}}.
        \end{array} \right.\]
\end{itemize}

\noindent Bence Csajb\'ok\\
MTA--ELTE Geometric and Algebraic Combinatorics Research Group\\
ELTE E\"otv\"os Lor\'and University, Budapest, Hungary\\
Department of Geometry\\
1117 Budapest, P\'azm\'any P.\ stny.\ 1/C, Hungary\\
{{\em csajbokb@cs.elte.hu}}

\bigskip

\noindent Giuseppe Marino\\
Dipartimento di Matematica e Fisica,\\
Universit\`a degli Studi della Campania ``Luigi Vanvitelli'',\\
Viale Lincoln 5, I-\,81100 Caserta, Italy\\

\noindent Dipartimento di Matematica e Applicazioni ``Renato Caccioppoli"\\
Università degli Studi di Napoli ``Federico II",\\
Via Cintia, Monte S.Angelo I-80126 Napoli, Italy\\
{{\em giuseppe.marino@unicampania.it}, {\em giuseppe.marino@unina.it}}

\bigskip

\noindent Olga Polverino and Ferdinando Zullo\\
Dipartimento di Matematica e Fisica,\\
Universit\`a degli Studi della Campania ``Luigi Vanvitelli'',\\
I--\,81100 Caserta, Italy\\
{{\em olga.polverino@unicampania.it}, \\{\em ferdinando.zullo@unicampania.it}}


\newpage

\begin{thebibliography}{pippo}
	
\bibitem{BGMP2015}
{\sc D. Bartoli, M. Giulietti, G. Marino and O. Polverino:}
Maximum scattered linear sets and complete caps in Galois spaces, Combinatorica {\bf 38}(2) (2018), 255–-278.

\bibitem{CSMP2016}
{\sc B. Csajb\'ok, G. Marino and O. Polverino:}
{Classes and equivalence of linear sets in $\mathrm{PG}(1,q^n)$}, J. Combin. Theory Ser. A {\bf 157} (2018), 402--426.

\bibitem{Delsarte}
{\sc P. Delsarte:}
Bilinear forms over a finite field, with applications to coding theory,
{\it J.\ Combin.\ Theory Ser.\ A} {\bf 25} (1978), 226--241.
	
\bibitem{DFH} {\sc U. Dempwolff, J. C. Fisher and A. Herman:} Semilinear transformations over finite fields are Frobenius maps, {\it Glasg. Math. J.} {\bf 42.2} (2000): 289--295.

\bibitem{Gabidulin}
{\sc E. Gabidulin:}
Theory of codes with maximum rank distance,
\emph{Problems of information transmission}, {\bf 21(3)} (1985), 3--16.

\bibitem{GQ2009x} {\sc R. Gow and R. Quinlan:} Galois theory and linear algebra,
{\it Linear Algebra Appl.} {\bf 430} (2009), 1778--1789.

\bibitem{GQ2009} {\sc R. Gow and R. Quinlan:} Galois extensions and subspaces of alterning bilinear forms with special rank properties,
{\it Linear Algebra Appl.} {\bf 430} (2009), 2212--2224.

\bibitem{McGS} {\sc G. McGuire  and J. Sheekey:} A Characterization of the Number of Roots of Linearized and Projective Polynomials in the Field of Coefficients, \href{https://arxiv.org/abs/1806.05853}{https://arxiv.org/abs/1806.05853}.

\bibitem{Lang} {\sc S. Lang:} Algebraic groups over finite fields, \emph{Amer. J. Math.} {\bf 78} (1956), 555--563.

\bibitem{LN} {\sc R. Lidl and H. Niederreiter:} Finite fields, {\it Cambridge university press,} Vol. {\bf 20}, 1997.

\bibitem{Ore} {\sc O. Ore:} On a special class of polynomials, {\it Trans. Amer. Math. Soc.} {\bf 35} (1933), 559--584.

\bibitem{Sh} {\sc J. Sheekey:} A new family of linear maximum rank distance codes, {\it Adv. Math. Commun.} {\bf 10}(3) (2016), 475--488.

\bibitem{PhDthesisFerd}
{\sc F. Zullo:} Linear codes and Galois geometries: between two worlds, {\it PhD thesis}, Universit\`a degli Studi della Campania ``Luigi Vanvitelli'' (2018).
\end{thebibliography}
\end{document}